\documentclass[a4paper,12pt]{amsart}
\usepackage{amsfonts}
\usepackage{amssymb}
\usepackage{ifthen}
\usepackage{graphicx}
\nonstopmode \numberwithin{equation}{section}
\usepackage[usenames]{color}
\newtheorem{thm}{Theorem}[section]
\newtheorem{lemma}[thm]{Lemma}
\newtheorem{cor}[thm]{Corollary}
\newtheorem{proposition}[thm]{Proposition}

\newcommand{\red}{\color{red}}

\newtheorem{cl}{Claim}[section]
\newtheorem{ca}{Case}[section]
\newtheorem{sca}{Subcase}[section]
\newtheorem{scl}[section]{Subclaim}
\newtheorem{conj}[equation]{Conjecture}

\theoremstyle{definition}
\newtheorem{defn}[thm]{Definition}
\newtheorem{op}[equation]{Open Problem}
\newtheorem{ques}[equation]{Question}
\newtheorem{rem}[thm]{Remark}
\newtheorem{exam}[equation]{Example}

\newcounter {own}
\def\theown {\thesection       .\arabic{own}}

\newenvironment{pf}[1][]{%
 \vskip 3mm
 \noindent
 \ifthenelse{\equal{#1}{}}%
  {{\slshape Proof. }}%
  {{\slshape #1.} }%
 }%
{\qed\bigskip}

\newcounter{alphabet}



\newcommand{\IR}{{\mathbb R}}

\newcommand{\id}{{\operatorname{id}}}

\newcommand{\diam}{{\operatorname{diam}}}

\newcommand{\dist}{{\operatorname{dist}}}

\def\be{\begin{equation}}
\def\ee{\end{equation}}

\newcommand{\ben}{\begin{enumerate}}
\newcommand{\een}{\end{enumerate}}

\newcommand{\blem}{\begin{lem}}
\newcommand{\elem}{\end{lem}}

\newcommand{\bthm}{\begin{thm}}
\newcommand{\ethm}{\end{thm}}
\newcommand{\bcor}{\begin{cor}}
\newcommand{\ecor}{\end{cor}}
\newcommand{\beg}{\begin{exam}}
\newcommand{\eeg}{\end{exam}}
\newcommand{\begs}{\begin{examples}}
\newcommand{\eegs}{\end{examples}}
\newcommand{\bdefe}{\begin{defn}}
\newcommand{\edefe}{\end{defn}}
\newcommand{\bprob}{\begin{prob}}
\newcommand{\eprob}{\end{prob}}
\newcommand{\bques}{\begin{ques}}
\newcommand{\eques}{\end{ques}}
\newcommand{\bei}{\begin{itemize}}
\newcommand{\eei}{\end{itemize}}
\newcommand{\bcon}{\begin{conj}}
\newcommand{\econ}{\end{conj}}
\newcommand{\bop}{\begin{op}}
\newcommand{\eop}{\end{op}}

\newcommand{\bca}{\begin{ca}}
\newcommand{\eca}{\end{ca}}
\newcommand{\bsca}{\begin{sca}}
\newcommand{\esca}{\end{sca}}

\newcommand{\bcl}{\begin{cl}}
\newcommand{\ecl}{\end{cl}}

\newcommand{\bscl}{\begin{scl}}
\newcommand{\escl}{\end{scl}}

\newcommand{\bcons}{\begin{conjs}}
\newcommand{\econs}{\end{conjs}}
\newcommand{\bprop}{\begin{propo}}
\newcommand{\eprop}{\end{propo}}
\newcommand{\br}{\begin{rem}}
\newcommand{\er}{\end{rem}}
\newcommand{\brs}{\begin{rems}}
\newcommand{\ers}{\end{rems}}
\newcommand{\bo}{\begin{obser}}
\newcommand{\eo}{\end{obser}}
\newcommand{\bos}{\begin{obsers}}
\newcommand{\eos}{\end{obsers}}
\newcommand{\bpf}{\begin{pf}}
\newcommand{\epf}{\end{pf}}
\newcommand{\ba}{\begin{array}}
\newcommand{\ea}{\end{array}}
\newcommand{\beq}{\begin{eqnarray}}
\newcommand{\beqq}{\begin{eqnarray*}}
\newcommand{\eeq}{\end{eqnarray}}
\newcommand{\eeqq}{\end{eqnarray*}}

\newcounter{minutes}
\divide\time by 60
\newcounter{hours}
\multiply\time by 60 \addtocounter{minutes}{-\time}

\begin{document}

\bibliographystyle{amsplain}
\title{Boundary rigidity of Gromov hyperbolic spaces}

\author{Hao Liang}
\address{Hao Liang, School of Mathematics and Big Data, Foshan University,  Foshan, Guangdong 528000, People's Republic
of China} \email{lianghao1019@hotmail.com}

\author{Qingshan Zhou${}^{\mathbf{*}}$}
\address{Qingshan Zhou, School of Mathematics and Big Data, Foshan University,  Foshan, Guangdong 528000, People's Republic
of China} \email{qszhou1989@163.com; q476308142@qq.com}

\def\thefootnote{}
\footnotetext{ \texttt{\tiny File:~\jobname .tex,
          printed: \number\year-\number\month-\number\day,
          \thehours.\ifnum\theminutes<10{0}\fi\theminutes}
} \makeatletter\def\thefootnote{\@arabic\c@footnote}\makeatother

\date{}
 \keywords{
Boundary rigidity, Gromov hyperbolic spaces, uniform perfectness, nonamenable, Cheeger isoperimetric constant, geodesically rich spaces.\\
${}^{\mathbf{*}}$ Corresponding author}

\begin{abstract} We introduce the concept of boundary rigidity for Gromov hyperbolic spaces. We show that a proper geodesic Gromov hyperbolic space with a pole is boundary rigid if and only if its Gromov boundary is uniformly perfect.  As an application, we show that for a non-compact Gromov hyperbolic complete Riemannian manifold or a Gromov hyperbolic  uniform graph, boundary rigidity is equivalent to having positive Cheeger isoperimetric constant  and also to being nonamenable. Moreover, several hyperbolic fillings of compact metric spaces are proved to be boundary rigid if and only if the metric spaces are uniformly perfect. Also, boundary rigidity is shown to be equivalent to being geodesically rich, a concept introduced by Shchur (J. Funct. Anal., 2013).
\end{abstract}

\thanks{Hao Liang was supported by NSFC (No.11701581 and No.11521101). Qingshan Zhou was partly supported by NSF of
China (No. 11901090), by Department of Education of Guangdong Province, China (No. 2021KTSCX116), by Guangdong Basic and Applied Basic Research Foundation (Nos. 2022A1515012441 and 2021A1515012289).}

\maketitle{} \pagestyle{myheadings} \markboth{Hao Liang and Qingshan Zhou}{Boundary rigidity of Gromov hyperbolic spaces}

\section{Introduction and main results}

An important theme of the theory of Gromov hyperbolic spaces is to determine when certain information of a hyperbolic space $X$ is completely captured by its Gromov boundary $\partial X$.  For example: When $X$ has a pole (see Definition \ref{d-p}), the quasi-isometry type of $X$ is completely captured by the quasisymmetry type of $\partial X$, see e.g. \cite{BS,BuSc,Jo10}. In this paper, we consider the following natural question: When can a self quasi-isometry $f$ of $X$ be determined up to a bounded error by its induced map $\partial f$ on $\partial X$? More precisely, we want to decide which Gromov hyperbolic spaces have the following property:

\bdefe[Boundary rigid]
A Gromov hyperbolic metric space $(X,d)$ is called {\em boundary rigid} if it has the following property: For any quasi-isometry $f:X\rightarrow X$ whose induced map on $\partial X$ is the identity map, we have $sup_{x\in X}d(x, f(x))< \infty$. We call $sup_{x\in X}d(x, f(x))$ the {\em displacement} of $f$.
\edefe

Boundary rigid has the following equivalent definition: Let ${\rm QI}(X)$ be the group of self quasi-isometries of $X$, i.e., ${\rm QI}(X)=\{f\mid f:X\rightarrow X \text { is a quasi-isometry}\}/\sim$, where $f\sim g$ if and only if $sup_{x\in X}d(f(x), g(x))< \infty$. It is known that every self quasi-isometry $f$ of a geodesic hyperbolic space $X$ induces a self homeomorphism $\partial f$ of $\partial X$ and quasi-isometries uniformly close to each other induce the same homeomorphism of $\partial X$, see \cite{BrHa,BuSc}. Therefore, there is a natural homomorphism $\partial : {\rm QI}(X)\rightarrow {\rm Homeo}(\partial X)$ which sends $f$ to $\partial f$.

\bdefe[Boundary rigid]
A Gromov hyperbolic metric space $X$ is called {\em boundary rigid} if the natural homomorphism $\partial : {\rm QI}(X)\rightarrow {\rm Homeo}(\partial X)$ is injective.
\edefe

\br As indicated in Proposition \ref{l-bg}, boundary rigidity of geodesic Gromov hyperbolic spaces is a quasi-isometric invariant.
\er

Let $\mathcal{PH}$ be the class of proper geodesic Gromov hyperbolic spaces with a pole.
$\mathcal{PH}$ includes many important examples of Gromov hyperbolic spaces: infinite hyperbolic groups \cite{Gr87}, Gromov hyperbolic domains in $\mathbb{R}^n$ \cite{BHK,ZR}, Gromov hyperbolic manifolds with a pole \cite{Cao,KeL,MR}, a lot of negatively curved solvable Lie groups \cite{D10,D14,SX,X12,X14}, and hyperbolic fillings \cite{BS,BuSc} etc. Our main theorem decides exactly which spaces in $\mathcal{PH}$ are boundary rigid.

\begin{thm}\label{thm-1}
Suppose that $X$ is a proper geodesic Gromov hyperbolic space with a pole. Then the following conditions are equivalent:
\begin{enumerate}
\item $X$ is boundary rigid.
\item $\partial X$ is uniformly perfect.
\end{enumerate}
\end{thm}

\br The implication $(2)\Rightarrow(1)$ in Theorem \ref{thm-1} follows from \cite[Theorem 2.3]{BP}. See \cite[Theorem 1.2]{Z} for a quantitative result. The harder part in Theorem \ref{thm-1} is the implication $(1)\Rightarrow(2)$  which is the main contribution of this paper.
\er

There are several equivalent definitions  for uniformly perfect spaces, see \cite{MR,S01}. The following was introduced by Mart\'inez-P\'erez and Rodr\'iguez in \cite{MR}.

\bdefe\label{def-1}
Let $S\geq 1$, A metric space $(X,d)$ is called {\it  $S$-uniformly perfect}, if there is some $r_0>0$ such that for each $x\in X$ and every $0<r\leq r_0$ there exists a point $y\in X$ such that $r/S< d(x, y)\leq r$. A metric space is called {\it uniformly perfect} if it is $S$-uniformly perfect for some $S\geq 1$.
\edefe

\br
The above definition of uniformly perfect spaces is equivalent to the standard definition (\cite{S01}) when $X$ is bounded.  Since we only consider this property in bounded spaces, all our results work as well with the standard definition, see \cite{MR}.
\er

Examples of uniformly perfect spaces include all connected spaces, many disconnected fractals such as the Cantor ternary set, Julia sets and the limit sets of nonelementary, finitely generated Kleinian groups of $\overline{\IR}^n$, see \cite{S01} and the references therein. Spaces with isolated points are not uniformly perfect.

Uniform perfectness is connected to several important concepts in function theory, geometry, group theory and geometric group theory. See \cite{Cao,MR,S01,Vu84}. Hence, by Theorem \ref{thm-1}, boundary rigidity is also connected to these concepts:

\begin{thm}\label{lz-3}
Assume that $X$ is a non-compact complete Riemannian manifold or a uniform graph which has bounded local geometry. If $X$ is Gromov hyperbolic and admits a pole, then the following conditions are equivalent:
\begin{enumerate}
\item $X$ is boundary rigid.
\item $\partial X$ is uniformly perfect.
\item $X$ has positive Cheeger isoperimetric constant.
\item $X$ is nonamenable.
\end{enumerate}
\end{thm}

\begin{proof} First, one observes from \cite[Proposition 5.2]{MR} that every non-compact complete Riemannian manifold with bounded local geometry is quasi-isometric to a uniform graph. It is known that the all properties mentioned in the theorem are all quasi-isometric invariants, such as Gromov hyperbolicity, uniformly perfectness, positive Cheeger isoperimetric constant and nonamenability, see \cite{MR}. Therefore, without loss of generality, we may assume that $X$ is a uniform graph.

The equivalence of (1) and (2) follows from Theorem \ref{thm-1}. The equivalence of (2) and (3) is recently established by Mart\'inez-P\'erez and Rodr\'iguez in \cite{MR}. The equivalence of (3) and (4) is a well known theorem of Folner, see \cite[Theorem 51]{CGH} or \cite[Proposition 2.3]{Ka}.
\end{proof}

\br For precise definitions of the concepts appeared in Theorem \ref{lz-3} we refer to \cite{Cao,MR}. The Cheeger isoperimetric constant is related to the work of Ancona \cite{A87,A88} about potential theory on hyperbolic manifolds and graphs. For such a space with boundary rigidity, the Dirichlet problem at infinity is solvable and the Martin boundary is homeomorphic to the Gromov boundary, see \cite{HLV,MR}. Amenability is an important property in the study of Brownian motion, harmonic analysis and random walks on graphs and manifolds, see \cite{Ka}. There are numerous equivalent conditions for this notation in geometric group theory, see \cite{CGH,Ko,NY} and the references therein.
\er

We now make several remarks on the requirement we put on the space $X$ in Theorem \ref{thm-1}:

\begin{rem}

\begin{enumerate}
  \item Since $\partial X$ equipped with all visual metrics are power quasisymmetrically equivalent to each other and since uniform perfectness is preserved by power quasisymmetries (\cite[Theorem 5.9 and Proposition 5.10]{MR}), one may state that $\partial X$ is uniformly perfect without specifying certain visual metric.

\item Without requiring that the space has a pole, it is easy to construct Gromov hyperbolic spaces with uniformly perfect boundaries that are not boundary rigid. Here is an example: Let $\mathbb{H}$ be the hyperbolic plane and $x\in\mathbb{H}$. Attach a sequence of segments $\{I_n=[0, n]\}_{n\in\mathbb{N}}$ to $x$ by identifying $0$ with $x$. Let $X$ be the resulting space endowed with the induced length metric. Then $\partial X=\partial \mathbb{H}=S^1$, which is uniformly perfect. However, $X$ is not boundary rigid since it admits quasi-isometries to itself that fix $\mathbb{H}$ and ``stretch'' and ``shrink''  the attached segments. In fact, it is clear from the above construction that every Gromov hyperbolic space with a uniformly perfect boundary can be embedded into a Gromov hyperbolic space with the same boundary that is not boundary rigid.

\item We expect Theorem \ref{thm-1} to be true if one only assumes that $X$ has quasi-geodesic between any two points or quasi-isometric to such a space. However for spaces with isolated points being very far apart, Theorem \ref{thm-1} is not true as the following example shows: Let $X=\{\pm x_n\mid x_n=2^{(2^n)},\;n\in \mathbb{Z}\}$ endowed with the Euclidean metric $|\cdot|$ of $\mathbb{R}$. Let $f:X\rightarrow X$ be a $(K, C)$-quasi-isometry such that its induced map $\partial f$ fixes $\partial X$. One finds that there exists $n_0$ depending only on $K$ and $C$, such that if $m> n\geq n_0$, then we have $|x_{m}-x_{m+1}|> K|x_n- x_{n+1}|+C$. One can show that $f$ has to fix $x_m$ for all $m>n_0$. Hence $X$ is boundary rigid but $\partial X$ is not uniformly perfect.
\end{enumerate}
\end{rem}

Boundary rigidity was previously investigated in \cite{ZR}, where the fact that the boundary being uniformly perfect implies  boundary rigidity was proved by the second author and Rasila. They were motivated by Teichmuller's displacement problem for the class of quasiconformal maps with identity boundary values and their proof uses tools from complex analysis. See  \cite{T,ZR} for the connections between boundary rigidity, Teichmuller's displacement problem and quasiconformal maps. We provide a different proof by using only Gromov hyperbolic geometry.

Boundary rigidity was also investigated in \cite{Sh}, where geodesically rich hyperbolic spaces  (see Definition \ref{gr}) were introduced and showed to be boundary rigid.  As an application of our main theorem we prove:

\begin{thm}\label{cor-111}
Let $X$ be a proper geodesic Gromov hyperbolic space with a pole. Then the following are equivalent:
\begin{enumerate}
  \item $X$ is boundary rigid.
  \item $\partial X$ is uniform perfect.
  \item $X$ is geodesically rich.
\end{enumerate}
\end{thm}

\cite[Question 1]{Sh} asks if a hyperbolic space can always be isometrically embedded into a geodesically rich hyperbolic space with an isomorphic boundary. While there are spaces without such embedding (See Section 4 for an example), the next corollary, together with Theorem \ref{cor-111}, completely characterizes all spaces admitting such an embedding.

\bcor\label{cor-112} Let $X$ be a proper geodesic Gromov hyperbolic space. If $X$ quasi-isometrically embeds into a geodesically rich Gromov hyperbolic space $Y$ with an isomorphic boundary, then $X$ is boundary rigid and $\partial X$ is uniformly perfect.
\ecor

A {\it hyperbolic filling} $\mathcal{H}(Z)$ of  a bounded complete metric space $(Z,d)$ is a Gromov hyperbolic space whose Gromov boundary (endowed with a visual metric) can be identified with the space $(Z,d)$ via a bilipschitz map. Hyperbolic fillings are useful in the study of asymptotic geometry and function space, see \cite{BoSa,Cao,MR}.  Several different constructions of hyperbolic fillings have been introduced (\cite{BS,BP,BuSc}), all of which produces a proper geodesic Gromov hyperbolic space with a pole provided $Z$ is compact. Hence Theorem \ref{thm-1} immediately yields:

\bcor Let $(Z,d)$ be a compact  metric space, and let $X$ be its hyperbolic filling in the sense of Bonk-Schramm \cite{BS}, Bourdon-Pajot \cite{BP}, and Buyalo-Schroeder \cite{BuSc}. Then $X$ is boundary rigid if and only if $Z$ is uniformly perfect.
\ecor

Here is the outline of this paper: In Section 2, we recall the definitions and results needed in the proof of Theorem \ref{thm-1}. Then we prove Theorem \ref{thm-1} in Section 3. In Section 4, we prove Theorem \ref{cor-111} and Corollary \ref{cor-112}.

\section{Preliminaries and auxiliary results}\label{sec-2}

In this section, we recall the basic definitions and some basic results that we need in the proof of Theorem \ref{thm-1}.

\bdefe
Let $f:(X,d)\to (Y,d')$ be a map between metric spaces $X$ and $Y$, and let $K\geq 1$ and $C\geq 0$ be constants. If for all $x,y\in X$ we have
$$K^{-1}d(x,y)-C\leq d'(f(x),f(y))\leq Kd(x,y)+C,$$
then $f$ is called a {\it $(K, C)$-quasi-isometric embedding}. If in addition, every point $y\in Y$ has distance at most $C$ from the set $f(X)$, then $f$ is called a {\it $(K, C)$-quasi-isometry}. In this case, there exists a quasi-isometry $g:(Y, d')\to (X, d)$ such that $f\circ g\sim \id_{Y}$ and $g\circ f\sim \id_{X}$. Moreover, $g$ is said to be an {\it inverse} of $f$.

We say that $X$ {\em quasi-isometrically embeds} into $Y$ if there is a quasi-isometric embedding from $X$ to $Y$. We say that $X$ is {\em quasi-isometric} to $Y$ if there is a quasi-isometry between $X$ and $Y$.
\edefe

\bdefe Let $I\subseteq \mathbb{R}$ be an interval. A curve $\gamma:I\to X$ is called a {\it $(K,C)$-quasi-geodesic} if $\gamma$ is a $(K,C)$-quasi-isometric embedding.
\edefe

\bdefe
A subset $E$ of $X$ is called {\em $M$-roughly full} if the $M$-neighborhood of $E$ in $X$ equals $X$.
\edefe

We now recall the definitions of Gromov hyperbolic spaces and their Gromov boundaries, see \cite{BHK,BS,BrHa,BuSc}. Let $(X,d)$ be a metric space. Fix a base point $w$ in $X$.
\begin{enumerate}

\item The space $X$ is called {\it geodesic}, if each pair of points $x,y\in X$ can be joined by a geodesic $[x,y]$; that is, a curve whose length is precisely the distance between $x$ and $y$.
\item
Suppose $(X,d)$ is geodesic. The metric space $X$ is called {\it $\delta$-hyperbolic} $(\delta\geq 0)$ if each point on the edge of any geodesic triangle is within distance $\delta$ of the other two edges. If $X$ is $\delta$-hyperbolic for some $\delta\geq 0$, then we say that it is Gromov hyperbolic.
\item
For $x,y\in X$, the {\it Gromov product} of $x,y$ with respect to $w$ is
$$(x|y)_w=\frac{1}{2}\big(d(x,w)+d(y,w)-d(x,y)\big).$$
\item
Suppose $(X, d)$ is $\delta$-hyperbolic.
\begin{enumerate}
\item
A sequence $\{x_i\}$ in $X$ is called a {\it Gromov sequence} if $(x_i|x_j)_w\rightarrow \infty$ as $i,$ $j\rightarrow \infty.$
\item
Two such Gromov sequences $\{x_i\}$ and $\{y_j\}$ are said to be {\it equivalent} if $(x_i|y_i)_w\rightarrow \infty$ as $i\to\infty$.
\item
The {\it Gromov boundary} $\partial X$ of $X$ is defined to be the set of equivalence classes of Gromov sequences, and $\bar X=X \cup \partial  X$ is called the {\it Gromov closure} of $X$.
\item
For $x\in X$ and $\xi\in \partial X$, the Gromov product $(x|\xi)_w$ of $x$ and $\xi$ is defined by
$$(x|\xi)_w= \inf \big\{ \liminf_{i\rightarrow \infty}(x|u_i)_w\mid\; \{u_i\}\in \xi\big\}.$$
\item
For $\xi,$ $\zeta\in \partial X$, the Gromov product $(\xi|\zeta)_w$ of $\xi$ and $\zeta$ is defined by
$$(\xi|\zeta)_w= \inf \big\{ \liminf_{i\rightarrow \infty}(x_i|y_i)_w\mid\; \{x_i\}\in \xi\;\;{\rm and}\;\; \{y_i\}\in \zeta\big\}.$$
\end{enumerate}
\end{enumerate}

Following \cite{BrHa}, we now recall the definition of the $\delta$-thin condition. Given any three positive numbers $a, b, c$, we can consider the metric tree $T (a, b, c)$ that has three vertices of valence one, one vertex of valence three, and edges of length $a$, $b$ and $c$. Such a tree is called a {\it tripod}. We call the valence three vertex the {\it center} of the tripod. Extend the definition of tripod in the obvious way to cover the cases where $a$, $b$ and $c$ are allowed to be zero and $\infty$. When $a,b,c$ are all $\infty$, we call $T (a, b, c)$ an {\it infinite tripod}.
Given any three points $x, y, z$ in a metric space, the triangle equality tells us that there exist unique non-negative numbers $a, b, c$ such that $d(x, y) = a + b$, $d(x, z) = a+c$ and $d(y,z) = b+c$; Note that $a = (y|z)_x$, $b = (x|z)_y$ and $c = (x|y)_z$. There is an isometry from $\{x, y, z\}$ to a subset of the vertices of $T(a, b, c)$ (the vertices of valence one in the non-degenerate case); we label these vertices $v_x,v_y,v_z$ in the obvious way.
Given a geodesic triangle, $\Delta=\Delta(x,y,z)$, we define the {\em comparison tripod} $T_{\Delta}$ of $\Delta$ by $T_{\Delta} :=T(a,b,c)$. The above map $\{x, y, z\}\rightarrow\{v_x, v_y, v_z\}$ extends uniquely to a map $\chi_{\Delta}: \Delta \rightarrow T_{\Delta}$ whose restriction to each side of $\Delta$ is an isometry.

\bdefe Let $\Delta$ be a geodesic triangle in a metric space $(X,d)$ and consider the map $\chi_{\Delta}: \Delta\rightarrow T_{\Delta}$ defined above. Let $\delta\geq 0$. Then
 $\Delta$ is said to be {\em $\delta$-thin} if $p,q\in \chi_{\Delta}^{-1}(t)$ implies $d(p,q) \leq \delta$, for all $t\in T_{\Delta}$.
\edefe

All triangles in a $\delta$-hyperbolic space are $\delta'$-thin for some constant $\delta'$ depending only on $\delta$ (See \cite[Proposition 1.17]{BrHa}). Enlarging $\delta$ if necessary, we assume all triangles in all $\delta$-hyperbolic spaces we consider in this paper are $\delta$-thin.

When $(X,d)$ is proper and geodesic, $\partial X$ can be identified with the geodesic boundary $\partial_G X$. By definition $\partial_G X$ is the set of all equivalence classes of geodesic rays, where two geodesic rays are {\it equivalent} if the Hausdorff distance between them is finite.

We now recall the definition of visual metrics on $\partial X$. See \cite{BrHa,BuSc} for more details.

\bdefe\label{a-1}
Let $X$ be a $\delta$-hyperbolic space with $\delta >0$ and $w\in X$ a base point. For $0<\varepsilon<\min\{1,\frac{1}{5\delta}\}$, define
$$\rho_{w,\varepsilon}(\xi,\zeta)=e^{-\varepsilon(\xi|\zeta)_w}$$
for all $\xi,\zeta\in \partial X$ with convention $e^{-\infty}=0$.

Let
$$d_{w,\varepsilon} (\xi,\zeta):=\inf\Big\{\sum_{i=1}^{n} \rho_{w,\varepsilon} (\xi_{i-1},\xi_i)\mid\;n\geq 1,\xi=\xi_0,\xi_1,\ldots,\xi_n=\zeta\in \partial X\Big\}.$$
Then $(\partial X, d_{w,\varepsilon})$ is a metric space and we call $d_{w,\varepsilon}$ the {\it visual metric} on $\partial X$ with respect to $w\in X$ and the parameter $\varepsilon$.
\edefe

The following is a useful estimate of the metric $d_{w, \varepsilon}$.

\begin{lemma}\label{l: visual metric} $($\cite{BuSc}$)$
For any $\xi,\zeta\in \partial X$, we have
$$\rho_{w,\varepsilon}(\xi,\zeta)/2\leq d_{w,\varepsilon}(\xi,\zeta)\leq \rho_{w,\varepsilon}(\xi,\zeta).$$
\end{lemma}

Now, we prove that boundary rigidity of geodesic Gromov hyperbolic spaces is preserved under quasi-isometries.

\begin{proposition}\label{l-bg} Let $(X,d)$ and $(Y,d')$ be geodesic $\delta$-hyperbolic spaces. If $\varphi:X\to Y$ is a $(\lambda,\mu)$-quasi-isometry and $Y$ is boundary rigid, then  $X$ is also boundary rigid.
\end{proposition}

\bpf Let $f:X\to X$ be a $(K,C)$-quasi-isometry with induced boundary map $\partial f=\id_{\partial X}$. Since $\varphi:X\to Y$ is a $(\lambda,\mu)$-quasi-isometry,  there is an inverse $\psi:Y\to X$ of $\varphi$ that is a $(\lambda', \mu')$-quasi-isometry, where $\lambda'$ and $\mu'$ depend only on $\lambda$ and $\mu$. Let
$$\widetilde{f}=\varphi\circ f\circ \psi:Y\to Y.$$
Then $\widetilde{f}$ is a quasi-isometry. Thus, $\widetilde{f}$ induces a boundary map
$$\partial \widetilde{f}=\partial(\varphi\circ f\circ \psi)=\partial\varphi\circ \partial f\circ \partial\psi=\partial\varphi\circ  \partial\psi=\id_{\partial Y}.$$

Since $Y$ is boundary rigid, there is an $M$ such that $d'(y, \widetilde{f}(y))\leq M$ for any $y\in Y$. Since $\psi$ is an inverse of $\varphi$, we have $d'(y,\varphi\circ\psi(y))\leq M'$ for some $M'$ independent of $y$. Since $\psi$ is a $(\lambda', \mu')$-quasi-isometry, for any $x\in X$, there exists some $y\in Y$ such that $d(x,\psi(y))\leq \mu'$.  Therefore, we have
\begin{eqnarray*} d(f(x),x)&\leq& d(f(x),f\circ\psi(y))+d(f\circ\psi(y),\psi(y))+d(x,\psi(y))
\\ \nonumber&\leq& K\mu'+C+\lambda d'(\varphi\circ f\circ\psi(y),\varphi\circ\psi(y))+\mu+\mu'
\\ \nonumber&\leq& K\mu'+C+\lambda (d'(\varphi\circ f\circ\psi(y), y)+d'(y, \varphi\circ\psi(y)))+\mu+\mu'
\\ \nonumber&\leq& K\mu'+C+\lambda (M+M')+\mu+\mu'.
\end{eqnarray*}
Hence $X$ is boundary rigid.
\epf

\bdefe\label{d-p}(\cite{MR}) Let $L\geq 0$. A proper geodesic space $X$ has an {\it $L$-pole} if there is  some point $o\in \bar X$ such that each point of $X$ lies in the $L$-neighborhood of some geodesic connecting $o$ to some point $v\in \partial X$. We call $o$ a pole of $X$. We say that $X$ has a {\it pole} if $X$ has an $L$-pole for some $L\geq 0$.
\edefe

\begin{rem}
The terms ``roughly starlike'' and ``quasi-pole" in the literatures \cite{BHK,Cao,MR} are used for spaces having a pole. In spaces where geodesics do not always exist, one can talk about similar properties using the concept of ``$L$-visual" (See \cite{BuSc}).
\end{rem}

The following lemma shows that $X$ has a pole in the space if and only if it has a pole in $\partial X$. Note that pole constants depend on the locations of poles in $X$.

\begin{lemma}\label{zh-0} $($\cite[Lemma 3.3]{Z}$)$ Let $X$ be a proper geodesic Gromov hyperbolic space whose boundary $\partial X$ contains at least two points. Then the following conditions are equivalent:
\begin{enumerate}
  \item\label{zh-1} $X$ has a pole $\xi\in \partial X$.
  \item\label{zh-2} Each point of $X$ is a pole.
  \item\label{zh-3} $X$ has a pole $w\in X$.
  \end{enumerate}
\end{lemma}

In the remaining part of this section, we assume that $(X,d)$ is a proper geodesic $\delta$-hyperbolic space. We cite some standard facts about Gromov hyperbolic spaces. See \cite{BM,BrHa,BuSc,Groves-Manning,Sh,Va05b} for proofs.

\begin{lemma}\label{l-0}
Let $o\in X$ and $x,y\in \bar X$. Then
$$\Big|d(o,[x,y])-(x|y)_o\Big|\leq 4\delta.$$
\end{lemma}

\begin{lemma}\label{l-1}
Let $K\geq 1$ and $C\geq 0$. Let $x, y\in \bar X$. Let $\gamma$ and $\beta$ be two  $(K, C)$-quasi-geodesics connecting $x$ and $y$. Then $\gamma$ is in the $M$-neighborhood of $\beta$, where $M=M(\delta, K, C)$ is a constant depending only on $\delta$, $K$ and $C$.
\end{lemma}

\begin{lemma}\label{l-2}
Let $\Delta$ be a geodesic triangle in $\bar X$.
\begin{enumerate}
\item Each edge of $\Delta$ is contained in the $3\delta$ neighborhood of the other two edges.
\item Suppose all vertices of $\Delta$ are in $\partial X$. Then there exist an infinitely tripod $T_{\Delta}$ and a map $f: \Delta\rightarrow T_{\Delta}$ satisfying:
\begin{enumerate}
\item $f$ is an isometry on each edge of $\Delta$.
\item $f^{-1}(a)$ has diameter at most $458\delta$ for any $a\in T_{\Delta}$.
\end{enumerate}
\end{enumerate}
\end{lemma}

See \cite[Lemma 2.11]{Groves-Manning} for a proof of the first statement and \cite[Lemma 6.25]{Va05b} for a proof of the second statement.

The proof of the following easy lemmas are left to the readers. Note that the constants may not be optimal.

\begin{lemma}\label{l:right angle}
Let $[x, y]\subset \bar X$ be a geodesic and $o\in X$. If $u$ is a point on $[x, y]$ closest to $o$, then $u$ is in both of the $10\delta$-neighborhoods of $[o, x]$ and $[o, y]$.
\end{lemma}

\begin{lemma}\label{I:right angle 2}
Let $\gamma\subset X$ be a geodesic and $o\in X$. If $v$ is a point on $\gamma$ closest to $o$, then for any $u\in \gamma$, 
$d(o, u)\geq d(o, v)+d(v, u)-6\delta.$
\end{lemma}

\bdefe\label{d-c}
Let $K \geq1$, $C\geq0$ and $\rho\geq0$. A point $x\in X$ is  called a {\em $(K , C, \rho)$-quasi-centroid} of a triple of points $\{x_1,x_2, x_3\}\subset \bar X$ if there are $(K , C)$-quasi-geodesics  $[x_1, x_2]$, $[x_2, x_3]$ and $[x_1, x_3]$ such that $x$ is within $\rho$ of each of them.
Let $\Delta_{(K , C, \rho)}(\{x_1,x_2, x_3\})$ be
the set of all $(K, C, \rho)$-quasi-centroids of $\{x_1,x_2, x_3\}$.
\edefe

The following facts are parts of \cite[Lemma 3.1]{BM} although \cite{BM} only considered centroids of triples in $X$ instead of $\bar X$.

\begin{lemma}\label{bound of centroid}
Suppose $\rho\geq 3\delta$.
\begin{enumerate}
\item For any $x_1, x_2, x_3\in \bar X$,  the set $\Delta_{(K , C, \rho)}(\{x_1, x_2, x_3\})$ has non empty intersection with any geodesic between any two points in $\{x_1, x_2, x_3\}$. In particular $\Delta_{(K , C, \rho)}(\{x_1, x_2, x_3\})\neq \emptyset$.
\item The diameter of $\Delta_{(K , C, \rho)}(\{x_1, x_2, x_3\})$ is bounded above by a constant $D=D(K, C, \rho,\delta)$ depending on $K, C, \rho,\delta$.
\end{enumerate}
\end{lemma}

\begin{proof}
Consider a geodesic triangle $[x_1,x_2]\cup[x_2,x_3]\cup[x_3,x_1]$.  By Lemma \ref{l-2}, there exists $w$ on $[x_1, x_2]$ such that $w$ is in the $3\delta$-neighborhood of $[x_2, x_3]$ and is also in the $3\delta$-neighborhood of $[x_1, x_3]$. Hence $w\in \Delta_{(K , C, \rho)}(\{x_1, x_2, x_3\})$ when $\rho\geq 3\delta$. Similarly, $\Delta_{(K , C, \rho)}(\{x_1, x_2, x_3\})$ also intersects $[x_2, x_3]$ and $[x_1, x_3]$ when $\rho\geq 3\delta$. Hence (1) follows.

We now prove (2). For all $x\in \Delta_{(K,C,\rho)}(\{x_1, x_2, x_3\})$, Lemma \ref{l-1} yields that $d(x,[x_i,x_j])\leq \rho+M$ for all $i\neq j\in \{1,2,3\}$, where $M$ depends only on $K$, $C$ and $\delta$. It follows from \cite[Lemma 11.87]{DK} that the collection of points within $\rho+M$ of each of the edges of an ideal triangle has diameter bounded by a constant depending only on $\rho+M$ and $\delta$.  Hence the diameter of $\Delta_{(K,C,\rho)}(\{x_1, x_2, x_3\})$ is also bounded by the same constant.
\end{proof}

\section{Proof of Theorem \ref{thm-1}}

The proof of Theorem \ref{thm-1} is divided into three propositions: Proposition \ref{l-b}, Proposition \ref{full implies perfect} and Proposition \ref{b-l}.

\begin{proposition}\label{l-b}
Suppose that $(X,d)$ is a proper geodesic $\delta$-hyperbolic space with a pole. If $\partial X$ is uniformly perfect, then $X$ is boundary rigid.
\end{proposition}

\bpf Let $K\geq 1$ and $C\geq 0$. Assume that $f:X\to X$ is a $(K,C)$-quasi-isometry with induced boundary map $\partial f=\id_{\partial X}$.
Let $o\in X$ be an $L$-pole of $X$. Let $x\in X$. Then $x$ is in the $L$-neighborhood of $[o, \xi]$ for some $\xi\in \partial X$.
By assumption, $(\partial X, d_{o, \epsilon})$ is uniformly perfect, where $d_{o, \epsilon}$ is the visual metric based at $o$ with some parameter $\epsilon$. Let $S$ and $r_0$ be the uniformly perfect constants. By uniform perfectness, there is a sequence of points $\{\xi_i\}_{i=1,\dots}\subset \partial X$ such that
\be\label{i e-1} r_0/S^i<d_{o, \epsilon}(\xi_i, \xi)\leq r_0/S^{i-1}.\ee

Hence we have

\be\label{i e-2} d_{o, \epsilon}(\xi_{i+2}, \xi_i)>\frac{r_0}{S^i}-\frac{r_0}{S^{i+1}}.\ee

By Lemma \ref{l: visual metric} and the general fact that $(x_1\mid x_2)_o\geq \min\{(x_1\mid x_3)_o, (x_3\mid x_2)_o\}-4\delta$ for all $x_1,x_2,x_3 \in \bar X$, (\ref{i e-1})  and  (\ref{i e-2}) imply:
\be\label{inequality es} (i-1)\frac{\ln{S}}{\epsilon}-\frac{\ln{2r_0}}{\epsilon} \leq (\xi_i\mid \xi)_o\leq i\frac{\ln{S}}{\epsilon}-\frac{\ln{r_0}}{\epsilon} \ee
and
\be\label{ine es -1} (i-1)\frac{\ln{S}}{\epsilon}-\frac{\ln{2r_0}}{\epsilon}-4\delta \leq (\xi_{i+2}\mid \xi_i)_o\leq (i+1)\frac{\ln{S}}{\epsilon}-\frac{\ln{r_0( S-1)}}{\epsilon}. \ee

Let $v_i$ be a point on $[\xi_i, \xi]$ closest to $o$. Lemma \ref{l-0} ensures that $$|d(o, v_i)-(\xi_i\mid \xi)_o|\leq 4\delta.$$ Let $v_{i, i+2}$ be a point on $[\xi_i, \xi_{i+2}]$ closest to $o$. By Lemma \ref{l-0}, it follows that $$|d(o, v_{i, i+2})-(\xi_{i+2}\mid \xi_i)_o|\leq 4\delta.$$  Hence we have:
\be\label{inequality es-1} (i-1)\frac{\ln{S}}{\epsilon}-\frac{\ln{2r_0}}{\epsilon}-4\delta \leq d(o, v_i)\leq i\frac{\ln{S}}{\epsilon}-\frac{\ln{r_0}}{\epsilon} +4\delta\ee
and
\be\label{ine es -1-1} (i-1)\frac{\ln{S}}{\epsilon}-\frac{\ln{2r_0}}{\epsilon}-8\delta \leq d(o, v_{i, i+2})\leq (i+1)\frac{\ln{S}}{\epsilon}-\frac{\ln{r_0(S-1)}}{\epsilon}+4\delta. \ee

By Lemma \ref{l:right angle}, both $v_{i+2}$ and $v_{i, i+2}$ are in the $10\delta$-neighborhood of $[o, \xi_{i+2}]$. Hence  (\ref{inequality es-1}) and (\ref{ine es -1-1}) imply that $d(v_{i+2}, v_{i, i+2})\leq D_1$, where $D_1$ depends on $S, r_0, \epsilon$ and $\delta$.

Let $u_i$ be a point on $[o, \xi]$ closest to $v_i$. Then  by Lemma \ref{l:right angle}, we have $d(u_i, v_i)\leq 10\delta$. Hence (\ref{inequality es-1}) implies that $d(u_{i}, u_{i+1})\leq D_2$ and $d(o, u_1)\leq D_2$, where $D_2$ depends on $S, r_0, \epsilon$ and $\delta$. Note that $\{u_i\}$ is a Gromov sequence with $\{u_i\}\in \xi$.

Let $u$ be a point on $[o, \xi]$ closest to $x$. Then $d(u, x)\leq L$. By the above properties of $\{u_i\}$, there is $u_i$ such that $d(u, u_i)\leq D_2$. Hence $d(u, v_i)\leq D_2+10\delta$, $d(u, v_{i+2})\leq 3D_2+10\delta$ and $d(u, v_{i, i+2})\leq 3D_2+10\delta+D_1$. It follows that $x$ is in the $D_3$-neighborhood of $[\xi_i, \xi]$, $[\xi_{i+2}, \xi]$ and $[\xi_i, \xi_{i+2}]$, where $D_3=3D_2+10\delta+D_1$ depends on $S, \epsilon, r_0, L$ and $\delta$. Hence we have
$$x\in \Delta_{(1,0,D_3)}(\{\xi_i,\xi_{i+2}, \xi\})\subset \Delta_{(K,C,KD_3+C)}(\{\xi_i,\xi_{i+2}, \xi\}).$$
Since  $f$ is a $(K,C)$-quasi-isometry with $f|_{\partial X}=\id_{\partial X}$,  we know that $f([\xi_i, \xi]),f([\xi_{i+2}, \xi])$ and $f([\xi_i,\xi_{i+2}])$ are $(K,C)$-quasi-geodesics, $f$ fixes $\{\xi_{i},\xi_{i+2},\xi\}$ point-wise and $f(x)$ is within $KD_3+C$ of $f([\xi_i, \xi]),f([\xi_{i+2}, \xi])$ and $f([\xi_i,\xi_{i+2}])$. This implies that
$$f(x)\in \Delta_{(K,C,KD_3+C)}(\{\xi_i,\xi_{i+2}, \xi\}).$$
Therefore, by Lemma \ref{bound of centroid}, we know that $d(x, f(x))$ is bounded above by a constant $D$ depending only on $K, C$ and $D_3$. 

\epf

Let $\Delta_{(K , C, \rho)}(\partial X)$
be the collection of all $(K , C, \rho)$-quasi-centroids of all triples in $\partial X$, i.e.,
$$\Delta_{(K , C, \rho)}(\partial X)= \bigcup_{\{\xi_1, \xi_2, \xi_3\}\subset \partial X}\Delta_{(K , C, \rho)}(\{\xi_1, \xi_2, \xi_3\})$$

\begin{proposition}\label{full implies perfect}
Suppose that $(X,d)$ is a proper geodesic $\delta$-hyperbolic space. If $\Delta_{(K , C, \rho)}(\partial X)$ is $M$-roughly full in $X$ for some $M,C,\rho\geq0$ and $K \geq1$, then $\partial X$ is uniformly perfect.
\end{proposition}

\begin{proof} 
Suppose on the contrary that $(\partial X,d_{p,\varepsilon})$ is not uniformly perfect, where $d_{p,\varepsilon}$ is the visual metric based at $p$ with some parameter $\epsilon$. Then for any $S>1$, there exists $\xi\in\partial X$ with the following property: For any $\epsilon_0$, there exists $\epsilon<\epsilon_0$ such that there is no $\xi_0\in \partial X$ satisfying $\epsilon/S<d_{p,\varepsilon}(\xi, \xi_0)<\epsilon$. Hence there are $t$ and $s$ on the geodesic ray $[p, \xi]$ such that $t\in[p,s]$ and for any $\xi'\in \partial X$, we have \be\label{z-1.3} (\xi\mid \xi')_p\leq d(p, t)\;\;\;\;\mbox{or}\;\;\;\;(\xi\mid \xi')_p\geq d(p, s).\ee
Moreover $s$ and $t$ can be chosen so that $d(s, t)=\varphi(S)$ for some increasing function $\varphi$ of $S$ satisfying
$$\varphi(S)\rightarrow\infty\;\;\;\;\;\;\;\mbox{as}\; S\rightarrow \infty.$$

Let $u$ be the midpoint of a geodesic $[s,t]$. Since $\Delta_{(K , C, \rho)}(\partial X)$ is $M$-roughly full in $X$, there exist  $\{\xi_1, \xi_2, \xi_3\}\subset \partial X$ such that $u\in \Delta_{(K , C, \rho+M)}(\{\xi_1, \xi_2, \xi_3\})$. Now we consider two cases.

\textbf{Case a.} Suppose  $\xi\not\in \{\xi_1, \xi_2, \xi_3\}$.

Reindex $\{\xi_1, \xi_2, \xi_3\}$ if necessary, we see from (\ref{z-1.3}) that there are two subcases:
 \begin{enumerate}
\item $(\xi\mid \xi_1)_p\geq d(p, s)$ and $(\xi\mid \xi_2)_p\geq d(p, s)$;
\item $(\xi\mid \xi_1)_p\leq d(p, t)$ and $(\xi\mid \xi_2)_p\leq d(p, t)$.
\end{enumerate}

In the first subcase,
we have
$$(\xi_1\mid\xi_2)_p\geq \min\{(\xi\mid \xi_1)_p,(\xi\mid \xi_2)_p\}-2\delta\geq d(p, s)-2\delta.$$
Lemma \ref{l-0} yields
$$d(p, [\xi_1, \xi_2])\geq (\xi_1\mid\xi_2)_p-2\delta\geq d(p, s)-6\delta,$$
which implies that
$$d(u, [\xi_1, \xi_2])\geq d(u, s)-6\delta=\frac{1}{2}\varphi(S)-6\delta.$$
Hence we may choose $S$ big enough so that $d(u, [\xi_1, \xi_2])\geq \rho+M+1$. Contradiction.

We now consider the second subcase. Let $D=D(K, C, \rho, \delta)$ be the number given by Lemma \ref{bound of centroid}. We choose $S$ big enough so that $d(s, t)=\varphi(S)$ is much bigger than  $K, \rho, D$ and $M$. Without loss of generality, we assume $M>3\delta$.  As a result, $d(p, u)$ is much bigger than $d(p, t)$. Combining this with the fact that $(\xi\mid \xi_1)_p\leq d(p, t)$,  we know that $[t, s]$ is in the $3\delta$-neighborhood of $[\xi, \xi_1]$. Similarly, $[t, s]$ is in the $3\delta$-neighborhood of $[\xi, \xi_2]$.  This implies that in the $\frac{1}{2}\varphi(S)$-neighborhood of $u$, $[\xi, \xi_1]$ and $[\xi, \xi_2]$ stay in the $3\delta$-neighborhood of each other. Hence $u$ is at least $\frac{1}{2}\varphi(S)-3\delta-D-M$ away from $[\xi_1, \xi_2]$. This contradicts the fact that $u$ is in the $M$-neighborhood of $[\xi_1, \xi_2]$ as $\frac{1}{2}\varphi(S)-3\delta-D-M$ is much bigger than $M$.

\textbf{Case b.} Suppose $\xi \in \{\xi_1, \xi_2, \xi_3\}$.

Without loss of generality, we may assume $\xi=\xi_3$. By the arguments in \textbf{Case a}, we only need to consider the subcase:
$$(\xi\mid \xi_1)_p\leq d(p, t)\;\;\;\;\mbox{and}\;\;\;\;(\xi\mid \xi_2)_p\geq d(p, s).$$
Since $(\xi\mid \xi_2)_p\geq d(p, s)$ and $\frac{1}{2}\varphi(S)$ is much bigger than $8\delta$, we easily find that
$$d(u,[\xi,\xi_2])>\frac{1}{2}\varphi(S)-8\delta.$$
Hence we can choose $S$ big enough so that $d(u, [\xi_1, \xi_2])\geq \rho+M+1$. Contradicting with $u\in \Delta_{(K , C, \rho+M)}(\{\xi_1, \xi_2, \xi\})$.
\end{proof}

The aim of the rest of this section is to prove the following proposition, which, together with Proposition \ref{full implies perfect} and Proposition \ref{l-b}, completes the proof of Theorem \ref{thm-1}.

\begin{proposition}\label{b-l}
Let $K\geq 1$, $C\geq0$ and $\rho>3\delta$. Let $(X,d)$ be a proper geodesic $\delta$-hyperbolic space with a pole. If $X$ is boundary rigid, then
$\Delta_{(K , C, \rho)}(\partial X)$ is roughly full in $X$.
\end{proposition}

\bdefe
Let $[y, z]$ be a geodesic in $\bar X$. When $x\in X$, the {\em projection} of $x$ onto $[y, z]$, denoted by ${\rm proj}_{[y,z]}(x)$, is the set of all points on $[y,z]$ closest to $x$. When $x\in \partial X$, ${\rm proj}_{[y,z]}(x)$ is the set of $(1, 0, 3\delta)$-quasi-centroids of $\{x, y, z\}$.
\edefe

By Lemma \ref{l:right angle} and Definition \ref{bound of centroid} we have

\begin{lemma}\label{projection-centroid}
${\rm proj}_{[y,z]}(x)\subset \Delta_{(1, 0, 10\delta)}(\{x,y,z\})$.
\end{lemma}

Let $D_0$ be the bound on the diameter of $\Delta_{(1,0,10\delta)}(\{\xi_1,\xi_2, \xi_3\})$ given by Lemma \ref{bound of centroid}. Fix $D={\rm max}\{D_0, 10000\delta\}$ for the rest of this section.

\begin{lemma}\label{sequence}
Fix $K\geq 1$, $C\geq 0$ and $\rho>6\delta$. Let $o$ be an $L$-pole of $\bar X$. Suppose $\Delta_{(K , C, \rho)}(\partial X)$ is not $M$-roughly full in $X$ for any $M>0$. Then there exists a sequence of geodesic segments $\gamma_i=[y_i, z_i]$ in $X$ with the following properties:
\begin{enumerate}
\item Each $\gamma_i$ is a subsegment of a geodesic $[o, \xi_i]$ for some $\xi_i\in \partial X$;
\item Elements of $\{\gamma_i\}$ are pairwise disjoint with $d(\gamma_i,\gamma_j)\geq 10L+10\delta+1$ for $i\neq j$;
\item $100D\leq d(y_i,z_i)\to\infty$ as $i\to\infty$;
\item ${\rm proj}_{\gamma_i}(\partial X)=\cup_{\xi\in \partial X}{\rm proj}_{\gamma_i}(\xi)$ is contained in the $D$-neighborhood of $\{y_i, z_i\}$.

\end{enumerate}
\end{lemma}

\begin{proof} The construction of $\gamma_i=[y_i, z_i]$ is as follows.  Since $\Delta_{(K , C, \rho)}(\partial X)$ is not $M$-roughly full in $X$ for any $M>0$, there are closed balls $B_i=B(x_i,r_i)$ with radius $r_i\rightarrow \infty$ such that $B_i$ is disjoint from the set $\Delta_{(K , C, \rho)}(\partial X)$ for all $i\in \mathbb{N}_+$. Without loss of generality we may assume that $d(B_i, B_j)\geq 10L+10\delta+1$ for any $i\neq j$ and $r_i\geq 100D+L$.

Since $o$ is an $L$-pole of $\bar X$, for each $i$ there is a geodesic  $[o,\xi_i]$ connecting $o$ to $\xi_i\in \partial X$ such that
$$d(x_i,y_i)\leq L,$$
for some $y_i\in[o,\xi_i]$. Let $z_i$ a point in $[y_i,\xi_i]\cap \partial B_i$. Then we have
$$d(y_i,z_i)\geq r_i-L\to\infty, \;\;\;\mbox{as}\;i\to\infty$$
and $\gamma_i=[y_i,z_i]\subset B_i$.

Now we check that $\gamma_i$ satisfy the desired properties. Note that $(1)$ $(2)$ and $(3)$ follow directly from construction.

For $(4)$, fix $i\in \mathbb{N}_+$. For any $\xi\in\partial X$, if $\xi=o$ or $\xi_i$, then (4) holds because $\gamma_i\subset[o,\xi_i]$. Assume that $\xi\notin\{o,\xi_i\}$. By Lemma \ref{l-2}, there is a point $w\in [o,\xi_i]$ such that
$$\max\{d(w,[o,\xi]),d(w,[\xi,\xi_i])\}\leq 3\delta.$$
Hence $w\in \rm{proj}_{[o,\xi_i]}(\xi).$

By Lemma \ref{zh-0}, we may assume that $o\in \partial X$. Then $\rm{proj}_{[o,\xi_i]}(\xi)\subset \Delta_{(K , C, \rho)}(\partial X)$ and so $\gamma_i\cap \rm{proj}_{[o,\xi_i]}(\xi)=\emptyset$. Therefore, we have $w\not\in\gamma_i$. Without loss of generality we may assume that $\gamma_i\subset[o,w]$ and $y_i\in[o,z_i]$. Hence $y_i$ and $z_i$ are in the $3\delta$-neighborhood of $[o, \xi]$. It follows that $[y_i, \xi]$ is in the $3\delta$-neighborhood of $[o, \xi]$.  Therefore $z_i$ is in the $6\delta$-neighborhood of $[y_i, \xi]$. As a result we have
$$z_i\in \Delta_{(1,0,6\delta)}(\{\xi,y_i,z_i\}).$$
This, together with Lemma \ref{bound of centroid}, tells us that  $\Delta_{(1,0,6\delta)}(\{\xi,y_i,z_i\})$ is contained in the $D$-neighborhood of $z_i$. Since
$$\rm{proj}_{[y_i,z_i]}(\xi)=\Delta_{(1,0,3\delta)}(\{\xi,y_i,z_i\})\subset \Delta_{(1,0,6\delta)}(\{\xi,y_i,z_i\}),$$
(4) follows.
\end{proof}

Let $o\in \partial X$ be an $L$-pole of $X$ and let $[o, \xi]$ be a geodesic with $\xi\in \partial X$.  Let $y, z$ be points on $[o, \xi]$ such that $y\in [o,z]$ and $d(y, z)>100D$. Let $\gamma=[y, z]$ be a subarc of $[o, \xi]$. Suppose ${\rm proj}_{\gamma}(\partial X)$ is contained in the $D$-neighborhood of $\{y, z\}$.

Let $R\geq D$. Set
\be\label{t-4} X_y(R)=\{x\in X\mid\; {\rm proj}_{\gamma}(x)\subset B(y,R)\}\ee
and similarly,
$$X_z(R)=\{x\in X\mid\; {\rm proj}_{\gamma}(x)\subset B(z,R)\}.$$

In the following, we prove some technique lemmas about the above sets and then use these results to construct the desired quasi-isometries with increasing displacements.

\begin{lemma}\label{t-1}
$X_y(3D)\cap X_z(3D)=\emptyset$.
\end{lemma}
\bpf  Suppose on the contrary that there is some point $p\in X_y(3D)\cap X_z(3D)$. By the definitions of $X_y(3D)$ and $X_z(3D)$, there are $p_y$ and $p_z$ in ${\rm proj}_{[y,z]}(p)$ such that $d(y,p_y)\leq 3D$ and $d(z,p_z)\leq 3D$. Lemma \ref{bound of centroid} gives that $d(p_y,p_z)\leq D$. Thus
$$100D\leq d(y,z)\leq d(y,p_y)+d(p_y,p_z)+d(z,p_z)\leq 9D, $$
which is impossible.
\epf

\begin{lemma}\label{t-2} Let $w\not\in X_y(2D)\cup X_z(2D)$. Let $u$ be a point on $[y,z]$ closest to $w$. Then
 $u$ is also a point on $[o,\xi]$ closest to $w$.
\end{lemma}
\bpf Let $v$ be a point on $[o,\xi]$ closest to $w$. Suppose $v\not\in [y,z]$, otherwise the lemma follows easily. Without loss of generality we may assume that $v\in [o,y]$.  By Lemma \ref{I:right angle 2}, we have
\begin{eqnarray*}d(w,u) &\geq& d(w,v)+d(v,u)-6\delta
\\ \nonumber&=& d(w,v)+d(v, y)+d(y,u)-6\delta
\\ \nonumber&\geq&  d(w,y)+d(y,u)-6\delta.
\end{eqnarray*}
Note that $d(w, u)\leq d(w,y)$. Hence $d(y,u)\leq 6\delta\leq D$. Then Lemma \ref{bound of centroid} yields that $w\in X_y(2D)$, contradiction.
\epf

\begin{lemma}\label{t-3}
If $x\not\in X_y(3D)\cup X_z(3D)$, then $d(x,[y,z])\leq L+3\delta$.
\end{lemma}
\bpf
Since $o\in \partial X$ is an $L$-pole of $X$, there is $\xi'\in\partial X$ such that
$$d(x,p')\leq L,$$
where $p'$ is a point on $[o,\xi']$ closest to $x$. Let $\bar x$ and $p$ be points on $[y,z]$ closest to $x$ and $p'$, respectively. Suppose that $d(x,[y,z])=d(x,\bar x)>2\delta+L$, otherwise the desired estimate is obvious. Then we have
$$d(p',[y,z])\geq d(x,[y,z])-d(x,p')>2\delta.$$
Using \cite[Projection Lemma 3.2]{Bo} with $R=2\delta+L$, we obtain
$$d(\bar x,p)\leq 8\delta\leq D.$$
The assumption $x\not\in X_y(3D)\cup X_z(3D)$ implies that $p'\not\in X_y(2D)\cup X_z(2D)$. Hence by Lemma \ref{t-2}, we know that $\bar x$ and $p$ are points on $[o,\xi]$ closest to $x$ and $p'$, respectively.

By Lemma \ref{l-2}, it follows that there are  $q\in \rm{proj}_{[o,\xi]}(\xi')\cap [o,\xi]$ and $q'\in \rm{proj}_{[o,\xi']}(\xi)\cap [o,\xi']$ with $d(q,q')\leq 458\delta$. Then there are two cases at our hands: $p'\in[q',\xi']$ and $p'\in [o,q']$. Indeed, we shall show that $p'\in[q',\xi']$ does not occur.

Assume that $p'\in[q',\xi']$. To get a contradiction, we first establish the following estimate
\be\label{z-1.4} d(p,q)\leq 1374\delta.\ee
If $p\in [q,\xi]$, then Lemma \ref{l-2} ensures that there are $p_0,q_0,p_0'\in [\xi,\xi']$ with $d(p,q)=d(p_0,q_0)$, $d(p',q')=d(p_0',q_0)$ and
$$\max\{d(p,p_0),d(q,q_0),d(q',q_0),d(p',p_0')\}\leq 458\delta.$$
Thus by the triangle inequality, we have
$$d(p,p')\geq d(p_0,p_0')-916\delta\geq d(p,q)+d(p',q)-1374\delta\geq d(p,q)+d(p',p)-1374\delta,$$
which implies (\ref{z-1.4}).

If $p\in [q,o]$, again by  Lemma \ref{l-2}, there is $v\in [o,q']$ with $d(p,v)\leq 458\delta$ and $d(p,q)=d(v,q')$. This guarantees that
$$d(p,p')\geq d(p',v)-458\delta\geq d(p',q)+d(p,q)-916\delta \geq d(p',p)+d(p,q)-916\delta,$$
and thus, (\ref{z-1.4}) holds.

Note that $d(y,z)>100D$ and $p\in[y,z]$ with $d(p,\{y,z\})>2D$. Now (\ref{z-1.4}) implies that $q\in [y,z]$ and so $q\in {\rm Proj}_{[y,z]}(\xi').$ Since ${\rm proj}_{\gamma}(\partial X)$ is contained in the $D$-neighborhood of $\{y, z\}$, one observes that
$$d(q,\{y,z\})<D.$$
On the other hand, we obtain
$$d(q,\{y,z\})\geq d(p,\{y,z\})-d(p,q)>2D-1374\delta>D,$$
contradiction.

Therefore, we know that $p'\in [o,q']$. Since $p$ is the closest point on $[o,\xi]$ to $p'$, Lemma \ref{l-2} ensures that $d(p,p')\leq 3\delta$. Hence
$$d(x,[y,z])\leq d(x,p)\leq d(x,p')+d(p,p')\leq L+3\delta.$$
\epf

\begin{lemma}\label{t} Let $x\not\in X_y(10D)\cup X_z(10D)$. Then we have the following:

\begin{enumerate}
  \item For all $x'\in X_y(3D)$, $d(y,[x,x'])\leq 5D$.
  \item For all $x''\in X_z(3D)$, $d(z,[x,x''])\leq 5D$.
\end{enumerate}
\end{lemma}

\begin{proof} We only need to prove $(1)$, because $(2)$ follows from a similar argument. Choose points $p_x$ and $p_{x'}$ on $[o, \xi]$ closest to $x$ and $x'$, respectively. We require $p_{x}$ to be outside of the $10D$-neighborhood of $y$. Then $p_x$ and $p_{x'}$ are at least $7D$ apart. Consider a geodesic quadrilateral $[x, p_x, p_{x'}, x']$. We claim that

\bcl\label{z-1.5} $[p_{x'}, x']\cup[p_{x'}, p_x]\cup[p_x, x]$ is contained in the $5\delta$-neighborhood of $[x', x]$. \ecl

Let $w_1$ and $w_2$ be the pre-images of the center of the comparison tripod of the geodesic triangle $[x', p_{x'}, p_x]$ in $[x', p_{x'}]$ and $[p_{x'}, p_x]$, respectively. Then $d(w_1, w_2)\leq \delta$. This implies that $d(w_1, p_{x'})\leq \delta$ since otherwise $d(x', w_2)< d(x', p_{x'})$, which contradicts the choice of $p_{x'}$. Hence $[p_{x'}, x']\cup[p_{x'}, p_x]$ is in the $2\delta$-neighborhood of $[x', p_x]$. Now consider the triangle $[x', p_x, x]$. Let $w_3$ and $w_4$ be the pre-images of the center of the comparison tripod of the geodesic triangle $[x, x', p_x]$ in $[x, p_{x}]$ and $[x', p_x]$, respectively.

We will prove by contradiction that
$$d(w_3, p_x)\leq 2\delta.$$
Suppose  $d(w_3, p_x)> 2\delta$. Let $w_5$ be a point on $[p_x, w_3]$ such that
$$2\delta<d(w_5, p_x)\leq 2\delta+1.$$
Then there is a point $w_6$ on $[p_x, x']$ such that $d(w_5, w_6)\leq \delta$ and $d(p_x, w_6)<2\delta+1$. Since $d(p_x, p_{x'})>5D>>2\delta$+1, there is a point $w_7$ in $[p_x, p_{x'}]$ such that $d(w_6, w_7)\leq \delta$. Hence we know that
\begin{eqnarray*}d(x, w_7) &\leq& d(x, w_5)+d(w_5, w_6)+d(w_6, w_7)
\\ \nonumber&\leq& d(x, w_5)+2\delta
\\ \nonumber&<&  d(x, w_5)+d(w_5, p_x)=d(x, p_x),
\end{eqnarray*}
contradicting the choice of $p_x$.  So $d(w_3, p_x)\leq 2\delta$.

Therefore $[x, p_x]\cup[p_x, x']$ is contained in the $3\delta$-neighborhood of $[x', x]$. Combining this with the fact that $[p_{x'}, x']\cup[p_{x'}, p_x]$ is in the $2\delta$-neighborhood of $[x', p_x]$, the claim follows.

Since $p_{x'}$ is in the $3D$-neighborhood of $y$, we obtain
$$d(y, [x, x'])\leq d(p_{x'}, y)+d(p_{x'}, [x, x'])\leq 3D+5\delta\leq 5D.$$
\end{proof}

Let $(X,d)$ be a proper geodesic $\delta$-hyperbolic space with a $L$-pole. Let $\gamma=[y, z]$ be a geodesic in $X$ such that ${\rm proj}_{\gamma}(\partial X)$ is contained in the $D$-neighborhood of $\{y, z\}$.
Define $\Phi_{[y,z]}:X\rightarrow X$ as follows.

\bdefe\label{d-q} Let $l>0$ and $\Phi_l: [0, l]\rightarrow [0, l]$ be defined by
$$
\Phi_l(t)=\begin{cases}
\displaystyle \; 2t,\;\;\;\;\;\;\;\;\;\;\;\;\;\;\;\;\;\;\;\;\;\;\;\;\;\;\;\;\;\;\;\;\text{if}\;\; t\in [0, l/3],\\
\displaystyle \; 2l/3+\frac{1}{2}(t-l/3), \;\;\;\;\;\;\;\;\,\text{if}\;\;  t\in [l/3, l].
\end{cases}$$

\begin{rem}
Note that $\Phi_l$ is a $(2, 0)$-quasi-isometry from $[0, l]$ to itself.
\end{rem}

Define $\Phi_{[y,z]}:X\rightarrow X$ as follows: If $x\in X_y(3D)\cup X_z(3D)$, then we define
$$\Phi_{[y, z]}(x)=x.$$
If $x\notin X_y(3D)\cup X_z(3D)$, then pick a point $p_x$ on $[y, z]$ closest to $x$ and we require that $p_x$ is not in the $3D$-neighborhood of $\{y, z\}$. Let $[y', z']$ be the largest subsegment of $[y, z]$ outside of the $3D$-neighborhood of $\{y, z\}$. Let $l=d(y', z')$ and let $\varphi: [y', z']\rightarrow [0, l]$ be an isometry. Then we define
$$\Phi_{[y, z]}(x)=\varphi^{-1}\circ\Phi_l\circ \varphi(p_x).$$
\edefe

Next we prove some properties of the map $\Phi_{[y, z]}$ for later use.
\begin{lemma}\label{map}
The map $\Phi_{[y, z]}$ satisfies:
\begin{enumerate}
\item The displacement of $\Phi_{[y, z]}$ is at least $l/3$.
\item $\Phi_{[y, z]}$ fixes $\partial X$.
\item $\Phi_{[y, z]}(X)$ is $(L+3\delta)$-roughly full in $X$.
\item $\Phi_{[y, z]}$ is a $(K', C')$-quasi-isometry with $K'$ and $C'$ depending only on $D,\delta,L$.
\end{enumerate}
\end{lemma}

\begin{proof}
$(1)$ By the definition of $\Phi_{[y, z]}$, the point on $[y',z']$ $l/3$ away from $y'$ is moved by $\Phi_{[y, z]}$ by $l/3$. Hence the displacement of $\Phi_{[y, z]}$ is at least $l/3$.

$(2)$ For all $x\in X$ with $d(x,[y,z])>L+3\delta$, Lemma \ref{t-3} yields that
$$x\in X_y(3D)\cup X_z(3D).$$
By the definition of $\Phi_{[y, z]}$, we see that $\Phi_{[y, z]}(x)=x$.  Hence $\Phi_{[y, z]}$ fixes $\partial X$.

$(3)$ Fix $x\in X$. If  $x\in X_y(3D)\cup X_z(3D)$, then $\Phi_{[y, z]}(x)=x$, as desired. Assume that $x\not\in X_y(3D)\cup X_z(3D)$. Let $p_x$ be the point on $[y, z]$ closest to $x$ such that $p_x$ is not in the $3D$-neighborhood of $\{y, z\}$. By the definition of the map $\Phi_{[y, z]}$, we observe that $\Phi_{[y, z]}$ is a self homeomorphism on $[y,z]$. Thus there is $\bar x\in [y,z]$ with $\Phi_{[y, z]}(\bar x)=p_x$. Hence Lemma \ref{t-3} guarantees that
$$d(x,\Phi_{[y, z]}(\bar x))=d(x,[y,z])\leq L+3\delta.$$

$(4)$ We now show that $\Phi_{[y, z]}$ is a quasi-isometry. Let
$$X_1=X_y(3D)\cup X_z(3D), \;\;\;\;\;X_2=X\setminus \Big(X_y(20D)\cup X_z(20D)\Big)$$
and let
$$X_3=X\setminus(X_1\cup X_2).$$
By the definition of $X_y(R)$ (see (\ref{t-4})), the sets $X_1, X_2$ and $X_3$ are pairwise disjoint that cover $X$. Let $x, x'\in X$. By (3), it suffices to show that
\be\label{z-1.6}
\frac{1}{K'}d(x, x')-C'\leq d(\Phi_{[y, z]}(x), \Phi_{[y, z]}(x'))\leq K'd(x, x')+C'.
\ee
By symmetry, there are four cases to check.

\bca Suppose $x, x'\in X_1$. \eca

Note that $\Phi_{[y, z]}$ is the identity map on $X_1$. So (\ref{z-1.6}) obviously holds for any $K'\geq 1$ and $C'\geq 0$.

\bca\label{ca-1}  Suppose $x, x'\in X_2\cup X_3=X\setminus X_1$.\eca
By Lemma \ref{t-3},  the map $\Phi_{[y, z]}$ sends $w\in X_2\cup X_3$ to $p_w$ and hence moves every point by at most $K_0=L+3\delta$. Hence it is a $(1, 2K_0)$-quasi-isometry. On the other hand, the map $\Phi_{[y, z]}$ is a $(2, 0)$-quasi-isometry on $[y',z']$.  Therefore in this case, the inequality (\ref{z-1.6}) is valid for $K'\geq 2$ and $C'\geq 4K_0$.

\bca\label{ca-2} Suppose $x\in X_1$ and $x'\in X_3$. \eca
Again by Lemma \ref{t-3}, $X_3$ is contained in the ($20D+K_0$)-neighborhood of $\{y, z\}$.  Thus by the definition of  $\Phi_{[y, z]}$, it is not difficult to see that
$$\Phi_{[y, z]}(X_3\cap B(y, 20D+K_0))\subset B(y, 40D)$$
and similarly,
$$\Phi_{[y, z]}(X_3\cap B(z, 20D+K_0))\subset B(z, 40D). $$
Hence it follows from Lemma \ref{t-3} that $\Phi_{[y, z]}$ moves $x'$ by at most $40D+K_0$. On the other hand, the map $\Phi_{[y, z]}$ fixes $x$. Hence we obtain (\ref{z-1.6}) for $C'\geq 40D+K_0$ and $K'\geq 1$.

\bca Suppose $x\in X_1$ and $x'\in X_2$. \eca
Let $[x, x']$ be a geodesic between $x$ and $x'$. Without loss of generality, we may assume that $x\in X_y(3D)$. Lemma \ref{t} shows that $[x, x']$ passes through the $5D$-neighborhood of $y$. Let $x''\in B(y, 5D)$ be a point on $[x, x']$. Note that $x\in X_1$ and $x''\in X_3$.

By Case \ref{ca-2}, the inequality (\ref{z-1.6}) holds for $x, x''$ with $C'=40D+K_0$ and $K'=1$.  By Case \ref{ca-1}, the inequality (\ref{z-1.6}) holds for $x'', x'$ with $C'=4K_0$ and $K'=2$. Since $x''\in[x,x']$, these two facts imply that
\begin{eqnarray*}d(\Phi_{[y, z]}(x), \Phi_{[y, z]}(x'))&\leq& d(\Phi_{[y, z]}(x), \Phi_{[y, z]}(x''))+d(\Phi_{[y, z]}(x''), \Phi_{[y, z]}(x'))
\\ \nonumber&\leq& d(x,x'')+40D+K_0 +2d(x',x'')+4K_0
\\ \nonumber&\leq&  2d(x, x')+5K_0+40D,
\end{eqnarray*}
as required.

We now prove the other half of (\ref{z-1.6}) for the points $x$ and $x'$. Consider a geodesic $[\Phi_{[y, z]}(x), \Phi_{[y, z]}(x')]=[x, \Phi_{[y, z]}(x')]$. By the definition of $\Phi_{[y, z]}$ and by the fact $x'\in X_2$, an elementary computation gives that
$$\Phi_{[y, z]}(x')\in X\setminus (X_y(10D)\cup X_z(10D)).$$
Since $x\in X_y(3D)$, Lemma \ref{t} shows that there is a point $x'''$ on $[x, \Phi_{[y, z]}(x')]$ such that
$$d(x''',y)\leq 5D.$$
This implies that
\be\label{z-2.0} d(x, x''')\geq d(x, x'')-d(y, x'')-d(y, x''')\geq d(x, x'')-10D.\ee
Since $y\in X_1$ and $x''\in X_3$, a similar argument as Case \ref{ca-2} gives that
$$ d(\Phi_{[y, z]}(x''),y)\leq d(\Phi_{[y, z]}(x''),x'')+d(x'',y)\leq 45D+K_0.$$
Since $x'\in X_2$ and $x''\in X_3$, the assertion in Case \ref{ca-1} yields that
\beq\nonumber\;\;\; d(\Phi_{[y, z]}(x'),x''')&\geq& d(\Phi_{[y, z]}(x''),\Phi_{[y, z]}(x'))-d(x''',y)- d(\Phi_{[y, z]}(x''),y)
\\ \nonumber &\geq& d(\Phi_{[y, z]}(x''),\Phi_{[y, z]}(x'))-50D-K_0
\\ \nonumber &\geq& \frac{1}{2}d(x'',x')-50D-5K_0.
\eeq
Hence this together with (\ref{z-2.0}) guarantees that
\beq\nonumber d(\Phi_{[y, z]}(x),\Phi_{[y, z]}(x'))&=& d(x,x''')+d(x''',\Phi_{[y, z]}(x'))
\\ \nonumber &\geq& d(x, x'')-10D+\frac{1}{2}d(x'',x')-50D-5K_0
\\ \nonumber &\geq& \frac{1}{2}d(x,x')-60D-5K_0,
\eeq
which completes the proof.
\end{proof}

\textbf{Proof of Proposition \ref{b-l}.} Assume that $(X,d)$ is a proper geodesic $\delta$-hyperbolic space with a pole. Suppose $\Delta_{(K , C, \rho)}(\partial X)$ is not $M$-roughly full in $X$ for any $M>0$. To prove that $X$ is not boundary rigid, it suffices to show that there exist $K_1$, $C_1$ and a $(K_1, C_1)$-quasi-isometry $\Phi:X\rightarrow X$ such that $\Phi$ fixes $\partial X$ but $\Phi$ has infinite displacement.

Let $\gamma_i=[y_i, z_i]$ be the sequence of segments given by Lemma \ref{sequence}. For each $i$, we may define the map $\Phi_i=\Phi_{[y_i, z_i]}$ as in Definition \ref{d-q}. We claim that
\bcl\label{z-1.7} Every point $x\in X$ is moved by at most one $\Phi_i$.\ecl
Suppose that there is some  $x$ moved by $\Phi_i$ and $\Phi_j$ with $i\neq j$. By the definition of $\Phi_i$, we obtain
$$x\in X\setminus(X_{y_i}(3D)\cup X_{z_i}(3D))\;\;\;\mbox{and}\;\;\; x\in X\setminus(X_{y_j}(3D)\cup X_{z_j}(3D)).$$
It follows from Lemma \ref{t-3} that
$$d(x,[y_i,z_i])\leq L+3\delta\;\;\;\mbox{and}\;\;\; d(x,[y_j,z_j])\leq L+3\delta,$$
which implies that
$$d(\gamma_i,\gamma_j)\leq 2L+6\delta.$$
On the other hand, by (2) of Lemma \ref{sequence}, we have
$$d(\gamma_i,\gamma_j)\geq 10L+10\delta+1,$$
contradiction. Hence Claim \ref{z-1.7} is true.

Now define
$$\Phi=\Phi_1\circ \cdots \circ \Phi_i\circ \cdots.$$
Claim \ref{z-1.7} ensures that $\Phi$ is well defined and moreover, $\Phi$ is a $(K_1, C_1)$-quasi-isometry with $K_1=K'^2$ and $C_1=K'C'+C'$ by using (4) of Lemma \ref{map}.

However, by (1) and (2) of Lemma \ref{map}, $\Phi$ fixes $\partial X$ and the displacement of $\Phi$ is not finite because $d(y_i,z_i)\to \infty$ as $i\to \infty$. This shows that $X$ is not boundary rigid.
\qed

\textbf{Proof of Theorem \ref{thm-1}.} It follows from Propositions \ref{l-b}, \ref{full implies perfect} and  \ref{b-l}.
\qed

\section{Geodesically rich spaces}

In \cite{Sh}, geodesically rich Gromov hyperbolic spaces were introduced and shown to be boundary rigid. The first goal of this section is to prove Theorem \ref{cor-111}, which says that, within the class of proper geodesic hyperbolic spaces with a pole, being boundary rigid is equivalent to being geodesically rich.  We now recall the definition of geodesically rich spaces and show that it can be simplified.

\bdefe\label{gr} A geodesic metric space $(X,d)$ is said to be {\it geodesically rich} if there are constants $r_0,r_1,r_2,r_3$, and $r_4$ such that
\begin{enumerate}
  \item for every pair of points $p$ and $q$ with $d(p,q)\geq r_0$, there exists a bi-infinite geodesic $\gamma$ such that $d(p,\gamma)<r_1$ and $|d(q,\gamma)-d(p,q)|<r_2$ and
  \item for any bi-infinite geodesic $\gamma$ and any $p\in X$, there exists a bi-infinite geodesic $\gamma'$ such that $d(p, \gamma')<r_3$ and $|d(p,\gamma)-d(\gamma', \gamma)|<r_4$.
\end{enumerate}
\edefe

\begin{lemma}\label{f implies s} Assume that $(X,d)$ is a proper geodesic $\delta$-hyperbolic space. If $X$ satisfies the first condition of Definition \ref{gr}, then $X$ is geodesically rich.
\end{lemma}

\begin{proof}
Assume $X$ satisfies the first condition of Definition \ref{gr} with constants $r_0$, $r_1$ and $r_2$.
Let $r_3=2r_1+r_2+r_0+1000\delta+1$ and  $r_4=2r_1+3r_2+r_0+1000\delta+1\geq r_3$. Let $\gamma$ be a bi-infinite geodesic and $p\in X$. We check that there exists a bi-infinite geodesic line $\gamma'$ such that  the second condition of Definition \ref{gr} holds with the above $r_3$ and $r_4$.

Let $q$ be a point on $\gamma$ closest to $p$. If $d(p,\gamma)<r_3$, then take $\gamma'=\gamma$. Thus $d(p,\gamma')<r_3$ and $d(p,\gamma)$ differs from the distance between $\gamma'$ and $\gamma$ by not more than $r_4$.

It remains to consider the case that $d(p,\gamma)\geq r_3\geq r_0$. Since $X$ satisfies the first condition of Definition \ref{gr}, for the points $p$ and $q$, we know that there exists a bi-infinite geodesic line $\gamma'$ such that
\be\label{l-1.1} d(p,\gamma')<r_1 \ee
and
\be\label{l-1.2} |d(q,\gamma')-d(p,q)|<r_2.\ee

We show that $\gamma'$ is the desired. To this end, let $p'$ and $q'$ on $\gamma'$ be such that $d(p,p')=d(p,\gamma')$ and $d(q,q')=d(q,\gamma')$. Let $q_0\in \gamma$ be such that $d(q',q_0)=d(q',\gamma)$. For later use, we need some auxiliary estimates. The first one is
\be\label{l-1.3} d(p',q')<r_1+r_2+6\delta. \ee

Indeed, by (\ref{l-1.1}) and (\ref{l-1.2}), we have $d(q, p')\leq d(q, p)+d(p, p')\leq d(q,\gamma')+r_1+r_2=d(q,q')+r_1+r_2.$
By Lemma \ref{l-2} and by the definition of $\delta$-hyperbolicity, we obtain
$d(q, p')\geq d(q,q')+d(p',q')-6\delta.$
These two inequalities imply (\ref{l-1.3}).

Secondly, we check
\be\label{l-1.4} d(q,q_0)<2r_1+2r_2+12\delta. \ee
The fact (\ref{l-1.1}) and (\ref{l-1.3}) guarantee that
\be\label{l-1.5} d(p,q')\leq d(p,p')+d(p',q')\leq 2r_1+r_2+6\delta.\ee
This, together with (\ref{l-1.2}), gives that
\beq\label{g-1.1} d(q, q')&=&d(q,\gamma')\leq  d(q, p)+r_2
\\ \nonumber&=& d(p,\gamma)+r_2
\\ \nonumber&\leq& d(q',\gamma)+d(p,q')+r_2
\\ \nonumber&\leq&  d(q',q_0)+2r_1+2r_2+6\delta.
\eeq
Again by Lemma \ref{l-2} and by the definition of $\delta$-hyperbolicity, we have
$d(q, q')\geq d(q_0,q')+d(q_0,q)-6\delta.$
Combining with (\ref{g-1.1}), the inequality (\ref{l-1.4}) follows.

Now by (\ref{l-1.5}) and by the choice of $r_3$, we have
$d(q',\gamma)\geq d(p,\gamma)-d(p,q')>100\delta.$
Then we may take $w\in [q,q']$ with $d(q',w)=20\delta$. Thus $d(w,\gamma)\geq 80\delta$ and $d(w,\gamma')=20\delta$. For any $x\in \gamma$ and $x'\in\gamma'$, let $[x,q]$ and $[x',q']$ be the sub-curves of $\gamma$ and $\gamma'$, respectively. Consider the geodesic quadrilateral $[x,q]\cup [q,q']\cup [q',x'] \cup [x',x]$. The above facts, together with the $\delta$-hyperbolicity of $X$, ensure that there is $w'\in[x,x']$ such that
$d(w,w')\leq 2\delta.$
Therefore, it follows from  Lemma \ref{l-2}, (\ref{l-1.4}) and (\ref{l-1.5}) that
\beq \nonumber d(x, x')&\geq& d(x, w')\geq d(x, w)-d(w,w')
\\ \nonumber&\geq& d(x,q')-d(w,q')-2\delta
\\ \nonumber&\geq& d(q_0,q')+d(q_0,x)-28\delta
\\ \nonumber&\geq& d(q,q')-d(q,q_0)+d(q_0,x)-28\delta
\\ \nonumber&\geq& d(q,q')-2r_1-2r_2-40\delta
\\ \nonumber&\geq& d(q,p)-2r_1-3r_2-40\delta.
\eeq

This yields $d(\gamma,\gamma')\geq d(p,q)-r_4$ by the choice of $r_4$.
On the other hand, the fact (\ref{l-1.2}) implies that
$d(\gamma,\gamma')\leq d(q,q')\leq d(p,q)+r_2\leq d(p, q)+r_4.$
Hence the lemma follows.
 \end{proof}

\begin{lemma}\label{l-p} Let $(X,d)$ be a proper geodesic $\delta$-hyperbolic space. If $X$ is geodesically rich, then $X$ has a pole.
\end{lemma}
\bpf
Let $X$ be a geodesically rich space with the constants
$r_0,r_1,r_2$ in the first condition.  Since $X$ is  proper and geodesic, any two points in $\bar X$ can be connected by a geodesic. Let $o\in X$ and $q\in X$.

Suppose first that $d(o,q)> r_0$. Since $X$ is geodesically rich, there exists a geodesic $\gamma$ joining $a,b\in\partial X$ such that $d(q,\gamma)< r_1$. Let $\alpha_1$ and $\alpha_2$ be geodesics joining $o$ to $a$ and $b$, respectively. Then by the $\delta$-hyperbolicity of $X$, we have $d(q,\alpha_1\cup \alpha_2)\leq r_1+\delta$.

For the other case that $d(o,q)\leq r_0$, we have
$d(q,\alpha_1\cup \alpha_2)\leq r_0.$
Hence $o$ is an $L$-pole of $X$, where $L={\rm max}\{r_1+\delta, r_0\}$.
\epf

\begin{proposition}\label{p: up implies gr}
Let $(X,d)$ be a proper geodesic $\delta$-hyperbolic space with a pole. If  $\partial X$ is uniformly perfect, then $X$ is geodesically rich.
\end{proposition}

\begin{proof}
By Lemma \ref{f implies s}, it suffices to check the first condition of Definition \ref{gr}. Let $D_0=D(1,0,3\delta,\delta)$ be the constants given in Lemma \ref{bound of centroid}. Since $\partial X$ is uniformly perfect, by Theorem \ref{thm-1}, $X$ is boundary rigid. Then by Proposition \ref{b-l},  $\Delta_{(1,0,3\delta)}(\partial X)$ are roughly $C$-roughly full in $X$ for some constant $C$. Let $C_1=C+3\delta$ and let $r_0=r_1=r_2= 1000D_0+C_1$.

Let $p$ and $q$ in $X$ with $d(p,q)\geq r_0$. Then $p$ is in the $C_1$-neighborhood of three bi-infinite geodesics $[\xi_1, \xi_2]$, $[\xi_3, \xi_2]$ and $[\xi_1, \xi_3]$, where $\xi_i\in \partial X$ for $i=1,2,3$. Let $q'$ be a point in $[\xi_1, \xi_2]\cup[\xi_3, \xi_2]\cup[\xi_1, \xi_3]$ closest to $q$. We consider two cases.

\bca Suppose $d(q' , p)<500D_0$.\eca
Let $\gamma$ be one of $[\xi_1, \xi_2]$, $[\xi_3, \xi_2]$ and $[\xi_1, \xi_3]$. Then we have
$$d(q, q')\leq d(q, \gamma)\leq d(q, q')+d(q', p)+d(p,\gamma)\leq d(q, q')+500D_0+C_1$$
and
$$d(q,q')-500D_0 \leq d(q,q')-d(q',p)\leq d(q,p)\leq d(q,q')+d(q',p)\leq d(q,q')+500D_0.$$
Hence
$|d(q, \gamma)-d(q, p)|\leq 1000D_0+C_1= r_2.$

\bca Suppose $d(q' , p)\geq 500D_0$.\eca

Lemma \ref{l-2} ensures that $q'$ is in the $3\delta$-neighborhood of only two of the three geodesics $[\xi_1, \xi_2]$, $[\xi_3, \xi_2]$ and $[\xi_1, \xi_3]$. Without loss of generality, suppose that $q'$ is in the $3\delta$-neighborhood of $\gamma_1=[\xi_3, \xi_2]$ and $\gamma_2=[\xi_1, \xi_3]$. Let $\gamma=[\xi_1, \xi_2]$. Let $q_0$ be a point in $\gamma$ closest to $q$.

If $d(q_0, p)\leq 3D_0$, then we obtain
$$|d(q, q_0)-d(q, p)|\leq d(p, q_0)\leq 3D_0$$
and therefore,
$$|d(q, q_0)-d(q, p)|=|d(q, \gamma)-d(q, p)|\leq3D_0\leq r_2.$$

We are thus left to assume that  $d(q_0, p)>3D_0$. Without loss of generality, we assume $q'\in\gamma_1$ because the case $q'\in\gamma_2$ follows similarly. Again by Lemma \ref{l-2}, $\gamma$ is in the $3\delta$-neighborhood of $\gamma_1\cup \gamma_2$. So there are two subcases at our hands.

\textbf{Subcase} $4.2(a).$ Suppose $q_0$ is in the $3\delta$-neighborhood of $\gamma_1$.

Let $q_1$ be a point on $\gamma_1$ that is at most $3\delta$ away from $q_0$. Since $[q, q']$ and $[q', q_1]$ diverge at $q'$, there is a point $v$ on $[q, q_1]$ such that $v$ is at most $5\delta$ away from $q'$. Since $d(q_0, q_1)\leq 3\delta$, the geodesic $[q, q_0]$ is in the $3\delta$-neighborhood of $[q, q_1]$. Hence there is a point $v_1$ on $[q, q_0]$ such that $d(v_1, v)\leq 3\delta$. Note that $[v_1, q_0]$ and $[q_1, v]$ are in the $3\delta$-neighborhoods of each other. Also $[q_1, v]$ and $[q', q_1]$ are in the $5\delta$-neighborhoods of each other. Hence  $[v_1, q_0]$ and $[q', q_1]$ are in the $8\delta$-neighborhoods of each other. Since $d(q', p)\geq 500D_0>100D_0$, we know that $[q', q_1]$ passes through the $C_1$-neighborhood of $p$. Hence $[v_1, q_0]$ passes through the $(C_1+8\delta)$-neighborhood of $p$. Since $D_0>C_1+8\delta$, $[v_1, q_0]$ and hence $[q, q_0]$ pass through the $D_0$-neighborhood of $p$.

Let $w\in [q, q_0]\cap B(p, D_0)$. Since $d(q_0, p)>3D_0$, we have $d(w, q_0)\geq 2D_0$. Since $D_0>C_1+8\delta$, we obtain
\beq \nonumber d(q, q_0)&=& d(q, w)+d(w, q_0)\geq d(q, w)+2D_0
\\ \nonumber&>& d(q, w)+D_0+C_1+8\delta
\\ \nonumber&\geq& d(q, w)+d(w, p)+d(p, \gamma)
\\ \nonumber&\geq& d(q, \gamma).
\eeq
This contradicts the choice of $q_0$.

\textbf{Subcase} $4.2(b).$  Suppose $q_0$ is in the $3\delta$-neighborhood of $\gamma_2$.

Let $q_1$ be a point on $\gamma_2$ that is at most $3\delta$ away from $q_0$. Let $q''$ be a point on $\gamma_2$ that is closest to $q$. Note that $q''$ can be chosen such that it is at most $3\delta$ away from $q'$. Hence
$d(p, q'')\geq 500D_0-3\delta,$
which is big enough to guarantee that $[q'', q_1]$ passes through the $C_1$-neighborhood of $p$.  Now we can get a contradiction by repeating the above argument with $q'$ replaced by $q''$ and $\gamma_1$ replaced by $\gamma_2$ as in Subcase $4.2(a)$. The proof is complete.
\end{proof}

\textbf{Proof of Theorem \ref{cor-111}.}
By Proposition \ref{p: up implies gr}, (2) implies (1). By \cite[Theorem 3]{Sh}, (1) implies (3). By Theorem \ref{thm-1}, (3) implies (2).
\qed

\cite[Question 1]{Sh} asks if a hyperbolic space can always be isometrically embedded into a geodesically rich hyperbolic space with an isomorphic boundary. The real line does not admit such an embedding: since all the bi-infinite geodesics of a hyperbolic space whose boundary consists of two points fellow travel, the first condition of geodesically rich can not be satisfied. Corollary \ref{cor-112} completely characterizes all geodeisc hyperbolic spaces admitting such an embedding.

\textbf{Proof of Corollary \ref{cor-112}.}
Let $f:X\rightarrow Y$ be a quasi-isometric embedding which induces a bijection $\partial f$ from $\partial X$ and $\partial Y$. It suffices to show that $f(X)$ is $M$-roughly full in $Y$.  Let $y\in Y$. Since $Y$ is geodesically rich, by (1) of Definition \ref{gr}, there is a bi-infinite geodesic $\gamma$ such that $y$ is in the $r_1$-neighborhood of $\gamma$. Let $\xi_1$ and $\xi_2$ on $\partial Y$ be the endpoints of $\gamma$. Since $\partial f: \partial X\rightarrow \partial Y$ is bijective, there are $\xi'_1$ and  $\xi'_2$ on $\partial X$ such that $\partial f(\xi'_i)=\xi_i$ for $i=1,2$.

Since $X$ is a proper geodesic Gromov hyperbolic space, we know that there is a geodesic $\gamma'$ connecting $\xi'_1$ and $\xi'_2$. Let $\gamma''=f(\gamma')$. Then $\gamma''$ is a quasi-geodesic whose constants depends only on $f$. Using Lemma \ref{l-1}, one observes that $\gamma''$ is in the $r'$-neighborhood of $\gamma$, where $r'$ depends on $f$ and $\delta$. Hence $y$ is in the $M=(r_1+r')$-neighborhood of $\gamma''=f(\gamma')\subset f(X)$. Therefore, we know that $X$ is quasi-isometric to $Y$.

By Lemma \ref{l-p}, $Y$ has a pole. So Theorem \ref{cor-111} ensures that $Y$ is boundary rigid.  The second statement now follows from Proposition \ref{l-bg}. 
\qed

Combining the above proposition with Theorem \ref{cor-111}, we have the the following:

\bcor  Let $X$ be a proper geodesic $\delta$-hyperbolic space with a pole. Then the following are equivalent:
\begin{enumerate}
\item $X$ quasi-isometrically embeds into a geodesically rich $\delta'$-hyperbolic space $Y$ inducing a bijection between $\partial X$ and $\partial Y$.
\item $X$ is geodesically rich.
\item $\partial X$ is uniformly perfect.
\item $X$ is boundary rigid.
\end{enumerate}
\ecor

\bigskip


\end{document}